\documentclass[12pt]{exam}
\usepackage{amsfonts}
\usepackage{amssymb,amsmath}
\usepackage{latexsym}
\usepackage{amsthm,amssymb}
\usepackage{amscd}
\usepackage{color}
\usepackage{ulem}
\usepackage{enumerate}
\usepackage{graphicx} 
\newtheorem{thm}[equation]{Theorem}

\newtheorem{lem}[equation]{Lemma}

\theoremstyle{definition}
\newtheorem{defn}[equation]{Definition}

\numberwithin{equation}{section}

\begin{document}

\title{Skewable matrices over  $\wedge^2 V$  applied to locally metric connections}

\author{Mihail Cocos \& Kent Kidman}

\maketitle
\begin{abstract}
	A necessary condition for a connection in a vector bundle to be locally metric is for its curvature matrix, which consists of $2$ forms, to be skew symmetric with respect to some local frame. In this paper we give a simple algorithm that can be used to decide when a matrix of $2$  forms is equivalent to a skew symmetric matrix.  We apply this algorithm to verify whether a "{\it full rank}" curvature connection is locally metric.
\end{abstract}

\section{Introduction}
The problem of finding metrics compatible with a connection in a vector bundle has been addressed by some mathematical physicists and mathematicians. See for example (\cite{MC},\cite{VZ1},\cite{VZ2},\cite{T}).\\ This is a problem motivated by gauge theory in physics and also interesting as a pure mathematical question. We will now briefly describe the differential geometry problem and abstract from it the algebra problem our paper is concerned with.\\
Let $E$ be a vector bundle of rank $m$ above the smooth $n$ dimensional manifold $M,$  and $D$ be  a connection in $E.$ It is a standard fact in differential geometry that if the connection $D$ is locally metric (see for example \cite{FT}) then there exist a local frame $$\sigma=(\sigma_1,\sigma_2, \cdots, \sigma_m)$$ in $E$ such that the curvature matrix of $D$ with respect to $\sigma$ is skew symmetric. It has been shown in (\cite{MC}) that the skew symmetry of the curvature matrix with respect to some frame is, in the two dimensional case, also sufficient for the local metrizability of the connection.  The curvature matrix $\Omega$ consists of differential $2$ forms defined on $M$ and hence at a point  $p \in M$ the curvature has elements of  $\wedge^2 T_pM=T_pM \wedge T_pM$ as its entries. The problem then is to be able to find out if there is a frame in which the curvature matrix is skew symmetric. Here we have to remind the reader that if $\sigma'=(\sigma_1',\sigma_2', \cdots, \sigma_m')$ is a different local frame of the bundle $E$ then the curvature matrix with respect to this new frame is related to the curvature matrix $\Omega$ by the equation 
\newpage
\begin{equation}
\label{curvtrnsf}\Omega'=U^{-1}\Omega U 
\end{equation}
where $U$ is the matrix with scalar entries defined by 
\begin{equation}
\label{framechange}
\sigma'=\sigma U.
\end{equation}

The equation (\ref{curvtrnsf}) is essential since in the case of a connection $D$ one can hardly expect its curvature matrix to be skew symmetric with respect to an arbitrary frame $\sigma'.$ Therefore starting with an arbitrary local frame $\sigma'$ we would like to be able find the scalar matrix $U$ as in (\ref{framechange}) that will generate our desired frame $\sigma.$ It is therefore necessary for us to address the pure algebraic question of weather a given matrix  $\Omega'$ with entries $2$ forms is equivalent  as in (\ref{curvtrnsf}) with a skew symmetric matrix $\Omega.$ Next section will clarify the algebra problem.

\section{The matrix problem}

\noindent Let $V$ be an $n$-dimensional real vector space and  $e=(e_1,e_2,..., e_n)$ a frame in $V$,  and let  $\wedge^2 V=V \wedge V$ denote the vector space of wedge products of vectors in $V$.

\begin{defn} 

An $n \times n$ matrix $\Omega$ with entries in  $\wedge^2 V$ is {\bf skewable} is there is a non-singular $n \times n$ real matrix $U$ such that $(U^{-1} \Omega U)^T=-U^{-1} \Omega U$, i.e. if $U^{-1} \Omega U$ is skew-symmetric.

\end{defn}

\begin{defn}

Let $S=\{S_1,S_2,...,S_m\}$ be a set of $n \times n$ real matrices.  Then $S$ is \underline{simultaneously skewable} if there exist a nonsingular real $n \times n$ matrix $U$ such that for each $1 \le k \le m$, $U^{-1}S_kU$ is skew symmetric.

\end{defn}

\begin{lem}
Let $S=\{S_1,S_2,...,S_m\}$ be a set of $n \times n$ real matrices. Suppose $$\Omega = w_1 S_1 + w_2 S_2 +...+w_m S_m,$$\ with  $w_k \in \wedge^2 V$ and for which the set $\{w_1, w_2, ..., w_m\}$ is linearly independent. Then $\Omega$ is skewable if and only if $S$ is simultaneously skewable.

\end{lem}

\proof  Suppose $S$ is simultaneously skewable, and that $U^{-1} S_k U$ is skew symmetric for ${1 \le k \le m}$, for a nonsingular real $U$.  Let $\Omega = w_1 S_1 + w_2 S_2 +...+w_m S_m$,\  $w_k \in V^{\wedge 2}$, where the set $\{w_1, w_2, ..., w_m\}$ is linearly independent.  Consider

\quad $U^{-1} \Omega U = U^{-1}( \sum_{k=1} ^m w_k S_k) U = \sum_{k=1}^m w_k U^{-1} S_k U$.

\noindent Each matrix in the sum is skew symmetric, so $U^{-1} \Omega U$ is skew symmetric as well.

Now suppose there is a nonsingular $U$ such that $U^{-1} \Omega U$ is skew symmetric, and that $\Omega = w_1 S_1 + w_2 S_2 +...+w_m S_m$,\   $w_k \in \wedge^2 V$, where the set $\{w_1, w_2, ..., w_m\}$ is linearly independent.  Then

\begin{equation}\label{*}U^{-1} \Omega U = \sum_{k=1}^m w_k U^{-1} S_k U \end{equation}

\noindent Since the left hand side of $(\ref{*})$ is skew symmetric, the right hand side is as well.  Consider the $(i,i)$ element of the left hand side of $(\ref{*})$.  This must be zero so

\quad $0=\sum_{k=1}^m w_k U^{-1} S_k U(i,i)$

\noindent and since the $w_k's$ are linearly independent, this means that $U^{-1} S_k U (i,i) = 0$ for all $i$.

\noindent  Now consider the $(i,j)$ element of the left hand side of $(\ref{*})$ where $i \ne j$.  This is equal to the negative of the $(j,i)$ element.  Hence

\quad $\sum_{k=1}^m w_k U^{-1} S_k U(i,j) = -\sum_{k=1}^m w_k U^{-1} S_k U(j,i)$

\noindent so, after rearranging, we have

\quad $\sum_{k=1}^m w_k ( U^{-1} S_k U(i,j) + U^{-1} S_k U(j,i)) = 0$,

\noindent and since the $w_k's$ are linearly independent, this means that $(U^{-1} S_k U(i,j) + U^{-1} S_k U(j,i) = 0$ for all $i \ne j$.  Putting these two things together says that each matrix $ U^{-1} S_k U$ is skew symmetric for each $k$, and that the set $S$ is simultaneously skewable.    $\Box$

\vspace{.5in}

Let $\Omega$ be an $n \times n$ matrix with entries in $\wedge^2 V$, and let $e=(e_1,e_2,..., e_n)$ be a frame in $V$.  Then 

\begin{equation} \Omega = \sum_{w_k} w_k S_k, \,\  k=1,2,...,\binom{n}{2} \end{equation}

\noindent where $\{w_k, k=1,2,...,n\}=\{e_i \wedge e_j, 1 \le i < j \le n\}$  is the standard basis for $\wedge^2 V$, and the $S_k$'s have real entries.  So we have

\begin{thm}

$\Omega$ as defined above, is skewable if and only if the matrices $\{S_k\}$ are simultaneously skewable.

\end{thm}

\vspace{10pt}





 Next we will give a simple algorithm that can be used to decide whether a set of matrices with real entries are simulataneoulsy skewable.

\begin{lem}\label{algorithm}
	Let $S=\{S_1,S_2,...,S_m\}$ be a set of $n \times n$ real matrices. Then $S$ is simultaneously skewable if and only if there exist a positive definite symmetric matrix $A$ such that
	
	\begin{equation} \label{system} S_iA+AS_i^T=0, \  \forall i.\end{equation}

\noindent If the solution to the system exist, then there is always a symmetric matrix $U$ such that 
		
	$$ U^2=A.$$ The matrix $U$ will skew-symmetrize the set $S.$
\end{lem}

\proof Suppose that $U$ is a nonsingular real matrix such that for each $1 \le i \le m$, $U^{-1}S_i U$ is skew symmetric.  Define $A=UU^{T}$.  Clearly $A$ is symmetric and positive definite.  Let $B= S_iA+AS_i^T = S_iUU^T + UU^TS_i^T$.  Then we have

\qquad $U^{-1}BU = U^{-1}S_iUU^TU + U^TS_i^TU$

\qquad\qquad $\ \ \ \ \ = (U^{-1}S_iU)U^TU + (S_iU)^TU$

\qquad\qquad $\ \ \ \ \ = (U^{-1}S_iU)U^TU + (UU^{-1}S_iU)^TU$

\qquad\qquad $\ \ \ \ \ =  (U^{-1}S_iU)U^TU + (U^{-1}S_iU)^TU^TU$

\qquad\qquad $\ \ \ \ \ = (U^{-1}S_iU + (U^{-1}S_iU)^T)U^TU$

\qquad\qquad $\ \ \ \ \ = 0$.

\noindent Hence $B=0$ and the first part is proved.

Conversely, suppose there is a positive definite symmetric matrix $A$ such that $S_iA+AS_i^T=0$ for each $i$.  Since $A$ is p.d.s., there is a orthogonal matrix $V$ such that 

\qquad\qquad $V^TAV = diag(\lambda_1, ... ,\lambda_n)$

\noindent with all the $\lambda_j$'s positive.  Then 

\qquad\qquad $A=V diag(\sqrt{\lambda_1}, ... ,\sqrt{\lambda_n})V^TV diag(\sqrt{\lambda_1}, ... ,\sqrt{\lambda_n})V^T$,

\noindent i.e. $A$ has a square root.  Define $U=\sqrt{A} = V diag(\sqrt{\lambda_1}, ... ,\sqrt{\lambda_n})V^T$.  Note that $U$ is symmetric.

\noindent We have the following:

\qquad \qquad $0=S_iA+AS_i^T = S_iU^2+U^2S_i^T$

\qquad \qquad $\ \  = S_iUU^T +UU^TS_i^T$

\qquad \qquad $\ \  = S_iUU^T +U(S_iU)^T$

\qquad \qquad $\ \  = UU^{-1}S_iUU^T +U(UU^{-1}S_iU)^T$

\qquad \qquad $\ \  = UU^{-1}S_iUU^T +U(U^{-1}S_iU)^TU^T$

\qquad \qquad $\ \  = U(U^{-1}S_iU +(U^{-1}S_iU)^T)U^T$.

\noindent Hence we have $U^{-1}S_iU +(U^{-1}S_iU)^T=0$, since $U$ is nonsingular. So $U^{-1}S_iU +U^{-1}S_iU)^T$ is skew symmetric for each $i$ and the set $S$ is simultaneously skewable.  $\Box$

\vspace{10pt}

Please note that the equations (\ref{system}) define an OVERDETERMINED( in general!) system of linear equations with the entries of the matrix $A$ as unknowns.

\section{Aplication to curvature matrices}

Let $V$ be $3$ dimensional real vector field and let $e^1,e^2,e^3$ be a frame in $V.$ Consider the "curvature" matrix 

$$ \Omega =\begin{Vmatrix}
-e^1 \wedge e^2-e^3 \wedge e^1 & -2 e^1 \wedge e^2- e^2 \wedge e^3- e^3 \wedge e^1 & -e^1\wedge e^2-2e^2\wedge e^3+e^3 \wedge e^1 \\ e^1\wedge e^2+e^3 \wedge e^1 & e^1\wedge e^2+e^2\wedge e^3+e^3 \wedge e^1 & 2e^2\wedge e^3 \\  -e^3\wedge e^1 & -e^2\wedge e^3- e^3 \wedge e^1 & -e^2 \wedge e^3
\end{Vmatrix}$$\\

\bigskip

It follows that $$S_1=\begin{Vmatrix}
-1 & -2 & -1\\ 1 & 1 & 0\\ 0 & 0 & 0
\end{Vmatrix} \, ,S_2=\begin{Vmatrix}
0 & -1 & -2\\ 0 & 1 & 2\\ 0 & -1 & -1
\end{Vmatrix}\, ,S_3=\begin{Vmatrix}
	-1 & -1 & 1\\ 1 & 1 & 0\\ -1 & -1 & 0 
\end{Vmatrix}$$

Using Mathematica, we find one solution to the system  $ S_iX+XS_i^T=0$, for $i=1,2,3$, is

\bigskip

\qquad \qquad $A = \begin{Vmatrix}
3 & -2 & 1\\ -2 & 2 & -1\\ 1 & -1 & 1
\end{Vmatrix}$.

\bigskip

\noindent $A$ is positive definite symmetric, so we use Mathematica again to find the square root of $A$, which is:

\bigskip

\qquad \qquad $U =  \begin{Vmatrix}
1.56022 & -0.689101 & 0.301417\\ -0.689101 & 1.17254 & -0.387684\\ 0.301417 & -0.387684 & 0.871119
\end{Vmatrix}$.

\bigskip

\noindent Note that we have rounded off for readability.  This leads to round off errors.  To get the results below, use more signigicant digits.

\noindent Then we see (counting for round off) the following:

\bigskip

\qquad $U^{-1}S_1U =  \begin{Vmatrix}
0 & -0.87112 & -0.387685\\ 0.87112 & 0 & -0.301416\\ 0.387685 & 0.301416 & 0
\end{Vmatrix}$

\bigskip

\qquad $U^{-1}S_2U =  \begin{Vmatrix}
0 & -0.086268 & -0.483435\\ 0.086268 & 0 & 0.871119\\ 0.483435 &-0.871119 & 0
\end{Vmatrix}$

\bigskip

\qquad $U^{-1}S_3U =  \begin{Vmatrix}
0 & -0.483435 & 0.784851\\ 0.483435 & 0 & 0.387685\\ -0.784851 &-0.387685 & 0
\end{Vmatrix}$

\bigskip

\noindent These are all skew-symmetric, which means the matrix $U$ skewsymmetrizes the set $\{S_1,S_2,S_3\}$, demonstrating the result.

\bigskip
 
 From the point of view of differential geometry, it is important to see weather two matrices that skewsymmetrize the same set of matrices are related. More precisely we will prove the following result 
 \begin{thm}\label{conjecture}
 	Let $U,V$ be two matrices with equal determinant that skewsymmetrize $\scriptstyle\binom{n}{2}$ linearly independent matrices $S_1,S_2,...,S_{\scriptstyle\binom{n}{2}}$  , i.e. $U^{-1}S_iU$ and $V^{-1}S_iV$ are skew-symmetric for each $i$. Then $$U=VO,$$ with $O$ orthogonal matrix.
 \end{thm} 
 In order to prove Theorem \ref{conjecture} we need to establish the following lemma
 
 \begin{lem}\label{conjlemma} Let $M_n^s$ denote the $\scriptstyle\binom{n}{2}$ dimensional space of $n \times n$ skew-symmetric matrices and let $A$ be a matrix with determinant equal to one. If for all $X \in M_n^s$ $$A^{-1}XA \in M_n^s,$$ then $A$ is an orthogonal matrix.

\end{lem}

\begin{proof}
By the Singular Value Decomposition Theorem we have \begin{equation}
A=PDL,
\end{equation}	

\noindent where $P$ and $L$ are orthogonal matrices($P,L \in \mathcal{O}(n)$)and $D$ is  the diagonal matrix  of singular values of $A$.\ We will first prove that \begin{equation}
\label{diagonal} D^{-1}XD \in M_n^s \,\, \,  \forall X \in M_n^s.\end{equation}

Assume the contrary, that there exist a skew symmetric matrix $X$ such that \begin{equation}\label{diagonal1}
D^{-1}XD \notin M_n^s.
\end{equation}
It follows that  \begin{equation}\label{diagonal2}
D^{-1}P^TPXP^TPD \notin M_n^s,
\end{equation}
and consequently 

 \begin{equation}\label{diagonal3}
L^TD^{-1}P^TPXP^TPDL \notin M_n^s,
\end{equation}
which in turn means 

 \begin{equation}\label{diagonal4}
A^{-1}PXP^TA \notin M_n^s
\end{equation}
for the skew symmetric matrix $PXP^T.$ Thus we contradict the hypothesis. It follows that equation (\ref{diagonal}) is true. Applying equation(\ref{diagonal}) to the special frame of $\scriptstyle\binom{n}{2}$ linearly independent skew symmetric matrices that have only two non zero entries we conclude that $$D=  I.$$ The conclusion of the lemma follows.

\end{proof}
	
Now we will prove Theorem \ref{conjecture}. 

\begin{proof}

Let $U,V$ be two matrices that skew-symmetrize $\scriptstyle\binom{n}{2}$ linearly independent matrices $S_1,S_2,...,S_{\scriptstyle\binom{n}{2}}$. Let us denote by $S$ the $\scriptstyle\binom{n}{2}$ dimensional span of $S_1,S_2,...,S_{\scriptstyle\binom{n}{2}}$ and let us note that both $U$ and $V$ induce the following linear isomorphisms \begin{equation}\begin{split}
\overline{U}: S \rightarrow M_n^s\\ \overline{U}=U^{-1}XU \end{split} \end{equation} and \begin{equation}\begin{split}
\overline{V}: S \rightarrow M_n^s\\ \overline{V}=V^{-1}XV \end{split} \end{equation} These two mappings are clearly linear and injective and by reasons of dimension it follows that both are indeed linear isomorphisms. It follows that $$\overline{U}\circ\overline{V}^{-1}:M_n^s \rightarrow M_n^s$$ is a linear isomorphism defined by the equation
\begin{equation}
\overline{U}\circ \overline{V}^{-1}(X)=U^{-1}VXV^{-1}U
\end{equation}

Since $$\det{U^{-1}V}=1$$ and by applying Lemma \ref{conjlemma} we conclude that  $\exists O \in \mathcal{O}(n)$ such that $$U=VO.$$
\end{proof}

 Theorem \ref{conjecture} allows us verify weather a a connection with "{\it full rank}" curvature matrix is locally metric. First we will clarify what we mean by  "{\it full rank}" curvature.
	
\begin{defn}

Let $D$ be connection in the vector bundle $E.$ Let $$\sigma=(\sigma_1,\sigma_2,...,\sigma_m)$$ be a local frame and let $\Omega$ be the curvature matrix with respect to $\sigma.$ Let $S_{ij}$ be the matrices defined by the equation \begin{equation}
\label{fullrank}\Omega=\sum_{i<j} (\sigma^i \wedge \sigma^j)  S_{ij}
\end{equation}
where $\sigma^i$ is the dual of $\sigma_i.$ The number $r$ of linearly independent matrices $(S_{ij})_{1 \leq i<j \leq m}$ will be called the rank of $\Omega.$ If $$r=\scriptstyle\binom{n}{2}$$ then $\Omega$ is said to have "{\it full rank}".

\end{defn}

\noindent It is trivial to verify that the rank of the curvature is a pointwise, scalar invariant of the connection.\\

 Now consider  a connection $D$ in the vector bundle $E.$ We assume that the curvature has "{\it full rank}" on a open set $U.$ We'd like to verify if the connection is metric on $U.$ Based on Theorem \ref{conjecture} and Theorem 1.2 in (\cite{MC}) we have the following steps in deciding weather the connection $D$ is metric on $U$:

\begin{enumerate}[(a)]
	\item Take a local frame $\sigma'=(\sigma'_1, \sigma'_2,...,\sigma'_m )$ and calculate the curvature matrix $\Omega$ with respect to this frame.
	\item Determine the matrices $S_{ij}$ defined in equation (\ref{fullrank}).
	\item Form the system $$XS_{ij}+S_{ij}^TX=0.$$
	\item If the linear system from the previous step doesn't have a symmetric, positive definite solution $X=A$, the $D$ is not metric.
	
	\item If the solution $A,$ exists then take its square root $\sqrt{A}=B$.
	\item Form the frame $\sigma$ defined by $$\sigma=\sigma' B$$ and test whether the connection forms with respect to $\sigma$ are skew. If they are not then the connection is not metric on $U$.
	\item If the connection forms with respect to $f$ are skew, then the metric that makes $\sigma$ orthonormal is compatible with the connection.
	
\end{enumerate}

\end{document}